\documentclass[11pt, twoside, a4paper]{article}
		\usepackage{amsmath,amssymb,amsthm,amscd,appendix,calrsfs,cite,color,dsfont,epsfig,eucal,enumitem,graphics,graphicx,latexsym,mathrsfs,mathtools,verbatim,hyperref}
	
		\newtheorem{theorem}{Theorem}
		\newtheorem{proposition}[theorem]{Proposition}
		
		\newtheorem{lemma}[theorem]{Lemma}
		
		\newtheorem*{teorema}{Theorem}
		
                \newcommand{\keywords}[1]{\par\addvspace\baselineskip\noindent\textbf{Keywords:}\enspace\ignorespaces#1}

                \newcommand{\AMSclassification}[1]{\par\addvspace\baselineskip\noindent\textbf{Mathematical subject classification:}\enspace\ignorespaces#1}
					
       \title{Exponential rate of decay of correlations  of equilibrium states associated with non-uniformly \\expanding circle maps}
       
        \author{Eduardo Garibaldi
        \thanks{Supported by FAPESP grant 2019/10485-8.} 
        \\ 
        \small{Department of Mathematics, University of Campinas, 13083-859 Campinas, Brazil} 
        \\ \small{(email: garibaldi@ime.unicamp.br)} 
        \\ ~ 
        \\ Irene Inoquio-Renteria 
        \thanks{Supported by Mathamsud Project TOMCAT 22-MATH-10.}
                 \\ \small{Instituto de Ciencias Físicas y Matemáticas, Universidad Austral de Chile,}
                \\ \small{casilla 567 Valdivia, Chile} 
         \\ \small{(email: ireneinoquio@uach.cl)}}
	
\begin{document}
			\maketitle
			
			\begin{abstract}
			In the context of expanding maps of the circle with an indifferent fixed point, understanding the joint behavior of dynamics and pairs of moduli of continuity $ (\omega, \Omega) $ may be a useful element for the development of equilibrium theory. 
			Here we identify a particular feature of modulus $ \Omega $ (precisely $ \lim_{x \to 0^+}  \sup_{\mathsf d} \Omega\big({\mathsf d} x \big) / \Omega(\mathsf d) = 0 $)
			as a sufficient condition for the system to exhibit exponential decay of correlations with respect 
			to the unique equilibrium state associated with a potential having $ \omega $ as modulus of continuity.
			This result is derived from obtaining the spectral gap property for the transfer operator acting on the space of observables with $ \Omega $ as modulus of continuity, a property that, as is well known, also ensures the Central Limit Theorem.
			Examples of application of our results include the Manneville-Pomeau family

			\AMSclassification{37D25, 37E05, 37D35, 37A25, 26A15}
			
			\end{abstract}

\section{Introduction and Main Results}

In the framework of uniformly hyperbolic systems, the study of statistical properties of equilibrium states for H\"older potentials has very well established theoretical foundations \cite{Bal00, Bow75, PP90, VO16, Rue04}.
This research has been extended to multiple different scenarios, and advances beyond a uniformly expanding setting may be founded, for instance, in \cite{LSV98, You98, LSV99, You99, FL01, Sar02, Go04a, Go04b, Hol05, CV13, LR14a, LR14b, Klo20}. 

Expanding maps on the circle that possess an indifferent fixed point are the center of attention of this note. 
Our core contribution is to provide sufficient conditions to show, without any induction intervention, exponential decay of correlations for these discrete time systems. 
Essentially we highlight how the regularity classes of potentials and observables should relate to the dynamics, as we will detail below when recalling from \cite{GI22} the notion of dynamic compatibility between moduli of continuity.  
The main novelty is the understanding of the role of the hypothesis according to which the module of the observables must vanish orderly.
Our assumptions allow thus to follow a typical strategy as in the consecrated way for Axiom~A diffeomorphism and  H\"older potentials:  we focus on proving quasicompactness and a spectral gap for the associated transfer operator. 
The considered conditions are at the same time comprehensive and flexible so that one can easily discuss the results in meaningful situations.

Here the phase space consists of the circle $ \mathbb T = [0, 1) $ equipped with the standard metric $ d(x, y) = \min \{|x-y|, |x-y\pm1|\} $.
The dynamics is described by a continuous map $ T $ on $ \mathbb T $ 
of the form $ T(x) = x(1 + V(x)) \mod  1 $, 
where $ V : [0, \infty) \to [0, \infty) $, with $ V(1) $ a positive integer,  
is demanded to be a continuous and increasing function that obeys the following regularly varying property\footnote{For information on these functions, see \cite{Sen76}.}:
$$ \exists \, \sigma \geq 0 \quad \textrm{s.t.} \quad \lim_{x\to 0} \frac{V(tx)}{V(x)} = t^\sigma, \quad \forall \, t > 0. $$
We do not impose additional constraints on the regularity of
$ V $. 

We consider in this work continuous potentials $ f : \mathbb T \to \mathbb R $ with a particular modulus of continuity $ \omega $, namely, potentials $ f $ such that
$$ |f|_\omega := \sup_{x \neq y} \frac{|f(x) - f(y)|}{\omega(d(x,y))} < \infty, $$
where the continuous function $ \omega : [0,  \infty) \to [0, \infty) $, with $ \omega(0) = 0 $, is supposed to be non-decreasing.
Let $ \mathscr C_\omega(\mathbb T)  $ denote the linear space of real-valued continuous functions on $ \mathbb T $ that admit $ \omega $ as a modulus of continuity.

Let $ C(\mathbb T) $ be the space of real-valued continuous functions on $ \mathbb T $ equipped with the uniform norm $ \| \cdot \|_\infty $.
The transfer operator associated with a potential $ f \in \mathscr C_\omega(\mathbb T) $ is the bounded linear operator that acts on $ C(\mathbb T) $  as
$$ \mathscr L_f\phi(x):=\sum_{y\in T^{-1}(x)}e^{f(y)} \phi(y), \quad \quad \forall \; x \in \mathbb T, $$
for a given $ \phi \in C(\mathbb T) $. As is well known, if it is possible to find a positive eigenfunction $ h $ for $ \mathscr L_f $ and an eigenmeasure $ d\nu $ for its dual $ \mathscr L_f^*$ 
(both with respect to a common positive maximal eigenvalue $ \chi $), taking into account normalization, it is to be expected that the probability $ d\mu = h \, d\nu $ is an equilibrium state of the system:
among $T$-invariant Borel probability measures $ m $, it maximizes the quantity $ h_m(T) + \int f \, dm $, where $  h_m(T) $ is the Kolmogorov-Sinai entropy of $ (\mathbb T, T, m) $.
When this is the case, by the variational principle, the maximal value, $ h_\mu(T) + \int f \, d\mu $, equals the topological pressure of the system, which we denote by $ P(T, f) $ (for details, see \cite{Wal82}). 

In \cite{GI22}, we have identified an appropriate linear space $ \mathscr C_\Omega(\mathbb T) $ to look at the action of the transfer operator $ \mathscr L_f $, $ f \in \mathscr C_\omega(\mathbb T) $,
in search of a positive eigenfunction, which has enabled us to prove existence and uniqueness of the Gibbs-equilibrium state for the system. Without inducing, a direct Ruelle-Peron-Frobenius theorem for a non-uniformly hyperbolic system was thus obtained. 
The key property was the $T$-compatibility of a modulus $ \Omega $ with respect to $ \omega $. 
In this note, we show how this property can also be useful for directly studying the spectrum of the transfer operator without coding the system, a strategy that for hyperbolic dynamics was crystallized by \cite{BKL02}.

Recall first that by a pre-orbit of  $ x_0 $ we mean any sequence
$ \{x_k\}_{k \ge 0} $ such that $ T(x_{k+1}) = x_k $ for all $ k $.  
We say that a modulus of continuity $ \Omega $ is $T$-compatible with respect to a given modulus $ \omega $ when
there are positive constants $ \varrho_1 $ and $ C_1 $ such that, for any points $ x_0 $ and $ y_0 $ with $ d(x_0, y_0) < \varrho_1 $, 
there exists a bijection among respective pre-orbits $ \{x_k\} $ and $ \{y_k\} $ fulfilling for all $ k $
\begin{eqnarray*}
	d(x_k, y_k) \le d(x_0, y_0) < \varrho_1 \textrm{ and } 
\end{eqnarray*}
\begin{equation}\label{compatibilidade}
	C_1 \sum_{j=1}^k \omega(d(x_j, y_j)) \le \Omega(d(x_0, y_0)) - \Omega(d(x_k, y_k)).
\end{equation}
When $ \Omega $ is concave, a sufficient condition for $T$-compatibility (see \cite[Proposition~7]{GI22})  is 
\begin{equation}\label{condicao_suficiente}
	\liminf_{x \to 0^+} \frac{V(x)}{\omega(x)}(\Omega((1+c)x) - \Omega(x)) > 0
\end{equation}
for every small enough $ c > 0 $.

In the present note, we point out an additional property of a concave modulus $ \Omega $ under which, with respect to the equilibrium state $ d\mu = h \, d\nu $, the system satisfies 
two significant statistical properties: exponential decay of correlations and the Central Limit Theorem. 
To do so, we follow a traditional method \cite{Bal00, Bow75, PP90, VO16, Rue04} that consists of showing that the transfer operator $ \mathscr L_f $ fulfills the spectral gap property,
more precisely, the rest of spectrum of $ \mathscr L_f $  lies inside a disc with radius strictly less than its maximal eigenvalue $ \chi $. 

The additional attribute to be respected by a concave modulus $ \Omega $ is the following limit
$$ \lim_{x \to 0^+} \; \sup_{0 < \mathsf d < 1/2} \frac{\Omega\big({\mathsf d} x \big)}{\Omega(\mathsf d)} = 0. $$
Whenever this limit occurs, for brevity we will say that $ \Omega $ \emph{vanishes orderly}.
The proposition below (whose proof is straightforward and will be omitted) indicates that, from compatible moduli, as long as concavity is ensured, a new dynamically compatible pair of moduli can be obtained in a canonical way so that the vanishing property holds.
We will use it to provide examples in the last section.

\begin{proposition}\label{ganho}
	Suppose that $V, $ $ \omega_0 $ and $ \Omega_0 $ are nonnegative continuous functions, with $ \omega_0 $ and $ \Omega_0 $ non-decreasing. 
	If the triple $ (V, \omega_0, \Omega_0) $ satisfies condition~\ref{condicao_suficiente}, then the triple 
	$ (V, \omega_s, \Omega_s) $, where $ \omega_s(x) := x^s \omega_0(x) $ and $ \Omega_s(x) := x^s \Omega_0(x) $ for $ s > 0 $,
	also satisfies condition~\ref{condicao_suficiente} and $ \Omega_s $ vanishes orderly.
\end{proposition}

We state our main results in the sequel.

\begin{theorem}[Spectral Gap Property] \label{lacunaespectral}
	Let $ \Omega $ be a $T$-compatible modulus of continuity with respect to $ \omega $. 
	Assume also that $ \Omega $ is concave and vanishes orderly.
	Then, concerning any potential $ f \in \mathscr C_{\omega}(\mathbb T) $,
	when acting on complex-valued continuous functions on $ \mathbb T $ that admit $ \Omega $ as modulus of continuity,
	the spectrum of $ \mathscr L_f $, other than the maximal eigenvalue $ \chi $, is contained in a disc with radius strictly smaller than this eigenvalue. 
\end{theorem}

It is very well established that the spectral gap property is not the big picture in the non-uniformly expanding context (see, for instance, \cite{You99, FL01, Sar02, Go04b, Hol05}).

A fundamental step in proving Theorem~\ref{lacunaespectral}, which is of its own importance (see  Proposition~\ref{desigualdadeDFLY}), is a 
Doeblin-Fortet-Lasota-Yorke-type inequality:
for the normalized potential $ \tilde f := f + \log h - \log h \circ T - \log \chi $, there exists a constant $ \tilde \Gamma > 0 $ such that
$$
\big\vert  \mathscr L_{\tilde f}^{n}  \phi \big\vert_\Omega \le \tilde \Gamma \, ( \tau(n) \, | \phi |_\Omega +  \vert\vert\phi\vert\vert_\infty ), \qquad 
\forall \, \phi \in  \mathscr C_{\Omega}(\mathbb T), \quad \forall \, n \ge 1.
$$
Here $ \tau(n) := \sup_{0 < \mathsf d < 1/2} \frac{\Omega\big(\theta(n) \, \mathsf d\big)}{\Omega(\mathsf d)} $, where 
$ \theta (n) :=  \frac{1}{\chi^n} \Big\vert \Big\vert  \mathscr L_{f - \log(1 + V)}^{n} \mathds 1 \Big\vert \Big\vert_\infty $ defines a sequence that tends to zero as $ n $ goes to infinity
whenever the hypothesis of $T$-compatibility is assumed.

Spectral gap implies exponential decay of correlations and the Central Limit Theorem.
In the following statements, all the hypotheses of Theorem~\ref{lacunaespectral} are implied.

\begin{theorem}[Exponential Decay of Correlations] \label{decaimentoexponencial}
	There exists $ \rho \in (0, 1) $ such that, given $ \phi, \psi \in \mathscr C_{\Omega}(\mathbb T) $, there is a positive constant $ K = K(\phi, \psi)  $ for which
	$$ \Big | \int \phi \,\, \psi \circ T^n \, d\mu - \int \phi \, d\mu \, \int \psi \, d\mu \Big | \le K \, \rho^n \qquad \forall \, n \ge 1. $$
\end{theorem}

As usual, we denote $ S_n \phi $ the $n$th Birkhoff sum $ \phi + \phi \circ T + \ldots + \phi \circ T^{n-1} $.

\begin{theorem}[Central Limit Theorem] \label{CLT}
	For any function $ \phi \in \mathscr C_{\Omega}(\mathbb T) $ which is not cohomologous to a constant, there is $ \gamma = \gamma(\phi) > 0 $ such that
	$$ \lim_{n \to \infty} \mu \Big\{  x \in \mathbb T : \frac{1}{\sqrt{n}} S_n\big(\phi - \int \phi \, d\mu \big)(x) < b \Big\} = \frac{1}{\gamma\sqrt{2\pi}} \int_{- \infty}^b e^{-t^2/2\gamma^2} \, dt. $$
\end{theorem}

The rest of this note is organized as follows: in the next section, we present the proofs of the main results, and in the last section we provide examples of applications.

\section{Proofs of the Main Results}

We consider a constant $ \varrho_V \in (0, 1/2) $ such that, if $ x, y \in [0,1] $ with $ | x - y | < \varrho_V $, then $  | x - y | N_V + |V(x) - V(y)| < 1/2 $,
where $ N_V := 1 + V(1) $  is the number of inverse branches of the map $ T $.  We also suppose that 
$$ \varrho_V < \frac{1}{2} \, \min_{0 \le i < N_V} d(a_{i+1}, a_i), $$
where $\{a_i\}_{i=0}^{N_V-1}$  is the (positively oriented) set of pre-images of $ a_0 := 0 =: a_{N_V} $. 
Finally we impose that the constant $ \varrho_V $ is small enough in order to ensure that 
$ \varrho_V + \varrho_V V(\varrho_V) + V(1) - (1-\varrho_V)V(1 - \varrho_V) < 1/2 $.
Denote by $ \mathsf{A_0} $ the set of pair of points in $ \mathbb T $ such that the origin belongs to the smallest open arc connecting them. 
Define 
$$ \lambda_V(x,y) := \mathds 1_{\mathsf{A_0}}(x,y) \, \min\{ V(x), V(y)\} + \mathds 1_{\mathsf{A_0}^\complement}(x,y) \, \max\{V(x), V(y)\}, $$
where  $ \mathds 1_{\mathsf S} $ represents the indicator function of a subset $ \mathsf S $. 
This terminology is useful to describe in the following lemma the non-uniformly expanding property  of the system.

\begin{lemma}
	If $ d(x, y) < \varrho_V $, then $ d(T(x), T(y)) \ge d(x, y) \, \big( 1 + \lambda_V(x,y) \big) $.
\end{lemma}

\begin{proof} Without loss of generality, we suppose $ 0 \le y \le x < 1$ in any situation.
	
	\smallskip
	
	\noindent \emph{Either $ (x,y) \notin \mathsf{A_0} $.} In this case, $ | x - y | = d(x, y)  < \varrho_V $ so that
	$$ \big |x  \big( 1+V(x) \big) - y  \big( 1+V(y)  \big) \big | \le |x-y| N_V + |V(x) - V(y)| < \frac{1}{2}. $$
	Therefore, as $ V $ is increasing, we see that

	\begin{eqnarray*} 
		d\big(T(x), T(y)\big) & =  x \big(1+V(x)\big) - y \big(1+V(y)\big) \\ & \ge (x - y) \big (1 + V(x) \big) = d(x, y) \big(1 + \max\{V(x), V(y)\} \big).
	\end{eqnarray*}
	
	\smallskip
	
	\noindent \emph{Or $ (x,y) \in \mathsf{A_0} $.}  Note thus that $ |x - y - 1| =  d(x, y) $ and the last condition required for the definition of $ \varrho_V $ guarantees
	$$ \big | x\big(1+V(x)\big)  - y\big(1+V(y)\big) -  \big(1 + V(1) \big) \big| \le | x - y - 1 | +  y V(y) + V(1) - xV(x) <  \frac{1}{2}.  $$
	Hence, we have that
	\begin{eqnarray*} 
		d\big(T(x), T(y)\big) & = y\big(1+V(y)\big) -  x\big(1+V(x)\big) + \big(1 + V(1) \big) \\
		& = d(x, y) + y V(y) - x V(x) +  V(1).
	\end{eqnarray*}
	Clearly $ V(1) \ge x V(x) + (1-x)V(y) $, which yields
	$$ y V(y) - x V(x) + V(1) \ge (y - x + 1)V(y) = d(x, y) V(y). $$ 
	We have shown that, in this case,
	$$  d\big(T(x), T(y)\big)  \ge d(x, y)  \big (1 + V(y) \big) = d(x, y) \big(1 + \min \{V(x), V(y)\} \big). $$
\end{proof}

The next result is an immediate consequence of the fact we dealing with order-preserving maps. We omit its simple proof. 

\begin{lemma}
	Suppose that $ d(x,y) < \varrho_V $. If $ \lambda_V(x,y) =  V(x) $, then
	$$ \prod_{i = 0}^{n-1} \Big( 1 + \lambda_V\big(T^i(x_n), T^i(y_n)\big) \Big) = \prod_{i = 0}^{n-1} \Big( 1 +V\big(T^i(x_n)\big) \Big), $$
	whenever $ x_n \in T^{-n}(x) $ and $ y_n \in T^{-n}(y) $ are a pair of pre-images such that  $ d(T^i(x_n), T^i(y_n)) < \varrho_V $ for $ i = 0, \ldots, n-1 $.
\end{lemma}

We may now present a version of a Doeblin-Fortet-Lasota-Yorke inequality. To that end, we make use of $\kappa_f := C_1^{-1}\vert f\vert_\omega$, where 
$ C_1$ is one of the constants that characterize the $T$-compatibility of $\Omega$ with respect to $\omega$ (see \ref{compatibilidade}). The other key constant, $ \varrho_1 $, can and
will be supposed smaller than $ \varrho_V $.

\begin{proposition}\label{desigualdadeDFLY}
	Let $ \Omega $ be a $T$-compatible modulus of continuity with respect to $ \omega $. Suppose that $ \Omega $ is concave.
	Given $ n \ge 1 $,  $\phi \in \mathscr C_\Omega(\mathbb T) $, and $x, y\in \mathbb T$ with $d(x,y)<\varrho_1$,  for $ \Gamma :=  \max \{2 \kappa_f e^{2 \kappa_f \; \Omega(1/2)}, \lceil \max h \, / \min h \rceil  \} $ 
	the following estimate holds
	$$
	\Big\vert  \mathscr L_{\tilde f}^{n}  \phi (x)   -  \mathscr L_{\tilde f}^{n} \phi (y)\Big\vert \le \Gamma \Big( | \phi |_\Omega \;  \Omega\big(\theta(n) \, d(x,y)\big) +  \vert\vert\phi\vert\vert_\infty  \; \Omega(d(x,y)) \Big),
	$$
	where $ {\displaystyle \theta(n) := \frac{1}{\chi^n} \Big\vert \Big\vert  \mathscr L_{f - \log(1 + V)}^{n} \mathds 1 \Big\vert \Big\vert_\infty} $. In particular,
	there exists a positive multiple $ \tilde \Gamma = \tilde \Gamma (\varrho_1) $ of the constant $ \Gamma $ such that
	$$
	\big\vert  \mathscr L_{\tilde f}^{n}  \phi \big\vert_\Omega \le \tilde \Gamma \big( \tau(n) \, | \phi |_\Omega +  \vert\vert\phi\vert\vert_\infty \big),
	$$
	with ${\displaystyle \tau(n) := \sup_{0 < \mathsf d < 1/2} \frac{\Omega\big(\theta(n) \, \mathsf d\big)}{\Omega(\mathsf d)}} $.
\end{proposition}

\begin{proof}
	Let $ \{x_k\}_{k \ge 1} $ be a pre-orbit of $x$. 
	By invoking the $T$-compatibility of $\Omega$ with respect to $\omega$, consider the unique pre-orbit $ \{y_k\}_{k \ge 1} $ of $y$ with respect to which the conditions in \ref{compatibilidade} are met.
	Note first that
	\begin{equation*}
		\big \vert S_n f(x_n)-S_n f(y_n) \big \vert \le \kappa_f \Big(\Omega\big(d(x,y)\big)-\Omega\big(d(x_n,y_n)\big)\Big).
	\end{equation*}
	Recall that  $ \tilde f= f + \log h - \log h \circ T - \log \chi $, where the function $ h $ belongs to 
	$$
	\left\{\psi\in  C^{0}(\mathbb T) : \, \psi\ge 0, \, \int \psi \,d\nu=1, \,\, \psi(x)\le \psi(y)\, e^{\kappa_f \, \Omega(d(x,y))} \textrm{ if } d(x,y)<\varrho_1\right\}.
	$$
	(For details, see the proof of Proposition~8 in \cite{GI22}.) In particular,
	\begin{equation*}
		\Big \vert \big( \log h - \log h \circ T^n \big)(x_n) - \big( \log h - \log h \circ T^n \big)(y_n) \Big \vert \le \kappa_f \Big(\Omega\big(d(x_n,y_n\big) + \Omega\big(d(x,y)\big)\Big),
	\end{equation*}
	so that one clearly obtains
	\begin{equation}\label{variation sum tilde f}
		\big \vert S_n \tilde f(x_n)-S_n \tilde f(y_n) \big \vert \le 2 \kappa_f \Omega\big(d(x,y)\big).
	\end{equation}
	
	From now on, without loss of generality, we assume the pair of points $ x $ and $ y $ is such that $ \lambda_V(x,y) = V(x) $.
	
	Note that 
	$$ 
	\Big\vert\mathscr L_{\tilde f}^n \phi (x)-\mathscr L_{\tilde f}^n \phi (y)\Big\vert \le  \sum_{(x_n, y_n)} \Big\vert e^{S_n \tilde f(x_n)}\phi(x_n)- e^{S_n \tilde f(y_n)}\phi(y_n) \Big\vert,
	$$
	where the sum is taken over pairs of pre-images $ (x_n, y_n) $ associated by the $T$-compatibility of $\Omega$ with respect to $\omega$.
	We have the following estimates
	\begin{eqnarray*}
		\Big\vert e^{S_n \tilde f(x_n)}\phi(x_n) - & e^{S_n \tilde f(y_n)}\phi(y_n) \Big\vert \le \\
		& \le e^{S_n \tilde f(x_n)} \, |\phi|_\Omega \, \Omega\big(d(x_n,y_n)\big) + \vert\vert\phi\vert\vert_\infty \, \Big\vert e^{S_n \tilde f(x_n)}-e^{S_n \tilde f(y_n)}\Big\vert \\
		& \le e^{S_n \tilde f(x_n)} \Big(  \, |\phi|_\Omega \, \Omega\big(d(x_n,y_n)\big) + \vert\vert\phi\vert\vert_\infty  \, \Big\vert e^{S_n \tilde f(y_n)- S_n \tilde f(x_n)} - 1\Big\vert \Big).
	\end{eqnarray*}
	From~\ref{variation sum tilde f}, we thus see that
	$$
	\Big\vert e^{S_n \tilde f(y_n)- S_n \tilde f(x_n)} - 1\Big\vert \le  2 \kappa_f e^{2 \kappa_f \; \Omega(1/2)} \, \Omega\big(d(x,y)\big) \le \Gamma \, \Omega\big(d(x,y)\big).
	$$
	Since $ \mathscr L_{\tilde f} \mathds 1 = \mathds 1 $ and $ \Omega $ is concave, we have
	\begin{eqnarray*}
		\Big \vert \mathscr L_{\tilde f}^n \phi (x)  -  \mathscr L_{\tilde f}^n \phi (y)\Big\vert 
		& \le  |\phi|_\Omega \,  \sum_{(x_n, y_n)}  e^{S_n \tilde f(x_n)}   \, \Omega\big(d(x_n,y_n)\big)  +   \Gamma \,  \vert\vert\phi\vert\vert_\infty  \, \Omega\big(d(x,y)\big) \\
		& \le  |\phi|_\Omega \,\,  \Omega\Big( \sum_{(x_n, y_n)}  e^{S_n \tilde f(x_n)}  d(x_n,y_n)\Big)  +   \Gamma \,  \vert\vert\phi\vert\vert_\infty  \, \Omega\big(d(x,y)\big).
	\end{eqnarray*}
	Thanks to the previous lemmas, 
	\begin{eqnarray*}
		d(x_n,y_n) & \le \frac{1}{\prod_{i = 0}^{n-1} \Big( 1 + \lambda_V\big(T^i(x_n), T^i(y_n)\big) \Big)} \, d(x, y) \\ & = \frac{1}{\prod_{i = 0}^{n-1} \Big( 1 +V\big(T^i(x_n)\big) \Big)} \, d(x, y).
	\end{eqnarray*}
	Hence, as $ \Omega $ is increasing, it follows that
	$$
	\Omega\Big( \sum_{(x_n, y_n)}  e^{S_n \tilde f(x_n)}  d(x_n,y_n)\Big)  \le \Omega\Big( \frac{\max h}{\min h} \, \frac{1}{\chi^n} \sum_{x_n \in T^{-n}(x)}  e^{S_n (f - \log(1+V))(x_n)}  \, d(x,y) \Big).
	$$
	Therefore, since $ \Omega $ is subadditive, we conclude that, whenever $d(x,y)<\varrho_1$
	$$
	\Big\vert  \mathscr L_{\tilde f}^{n}  \phi (x)   -  \mathscr L_{\tilde f}^{n} \phi (y)\Big\vert \le  \Gamma \Big( | \phi |_\Omega \;  \Omega\big(\theta(n) \, d(x,y)\big) + \vert\vert\phi\vert\vert_\infty  \; \Omega(d(x,y)) \Big),
	$$
	with $ \theta(n) = \frac{1}{\chi^n} \Big\vert \Big\vert  \mathscr L_{f - \log(1 + V)}^{n} \mathds 1 \Big\vert \Big\vert_\infty $.
\end{proof}

The fact that the sequence $ \{\theta(n)\} $ has a null limit is a consequence of the T-compatibility between moduli.

\begin{lemma}\label{teta}
	As $ n $ goes to infinite, $  \frac{1}{\chi^n} \Big\vert \Big\vert  \mathscr L_{f - \log(1 + V)}^{n} \mathds 1 \Big\vert \Big\vert_\infty $ tends to 0.
\end{lemma}

\begin{proof}
	For the normalized potential $ \tilde f= f + \log h - \log h \circ T - \log \chi $, it is easy to see that $  \mathscr L_{\tilde f - \log(1 + V)} \mathds 1  < \mathds 1 $ everywhere on $ \mathbb T $.
	Inductively it turns out that  for all $ n $
	$$  \mathscr L_{\tilde f - \log(1 + V)}^{n+1} \mathds 1  <   \mathscr L_{\tilde f - \log(1 + V)}^{n} \mathds 1.  $$  
	Consider then the pointwise limit function $ 0 \le g := \lim_{n \to \infty}  \mathscr L_{\tilde f - \log(1 + V)}^{n} \mathds 1 $. We claim that $ g \equiv 0 $.
	By Dini's theorem, the convergence of $ \{  \mathscr L_{\tilde f - \log(1 + V)}^{n} \mathds 1 \} $ to $ g $ is uniform. Thus clearly  $ \mathscr L_{\tilde f - \log(1 + V)} g = g $.
	If it held $ g(x) = \| g \|_\infty > 0 $ for some $ x \in \mathbb T $, then 
	$$ 0 <  \| g \|_\infty =  \mathscr L_{\tilde f - \log(1 + V)} g (x) \le  \| g \|_\infty \; \mathscr L_{\tilde f - \log(1 + V)} \mathds 1 (x)  $$ 
	would lead to the contradiction  $ \mathds 1(x) \le \mathscr L_{\tilde f - \log(1 + V)} \mathds 1(x) $.
	
	We have shown that  $ \Big\vert \Big\vert  \mathscr L_{\tilde f - \log(1 + V)}^{n} \mathds 1 \Big\vert \Big\vert_\infty  $ tends to 0.
	As $ \Big\vert \Big\vert  \mathscr L_{\tilde f - \log(1 + V)}^{n} \mathds 1 \Big\vert \Big\vert_\infty  \ge \frac{\min h}{\max h}  \, \frac{1}{\chi^n} \Big\vert \Big\vert  \mathscr L_{f - \log(1 + V)}^{n} \mathds 1 \Big\vert \Big\vert_\infty  $,
	the proof is complete.
\end{proof}

We are able now to prove the spectral gap property.

\begin{proof}[Proof of Theorem~\ref{lacunaespectral}]
	The spectrum of $ \mathscr L_{\tilde f} $ is the spectrum of $ \mathscr L_f $ scaled by $ \chi^{-1} $. 
	Therefore, if we denote by $ \mathbb C_\Omega^\perp $  the set of complex-valued continuous functions $ \phi $ on $ \mathbb T $ 
	such that their real and imaginary parts belong to $  \mathscr C_{\Omega}(\mathbb T) $ and $ \int \phi \, d\mu = 0 $, it suffices to 
	argue that the spectral radius of the restriction of $ \mathscr L_{\tilde f} $ to $ \mathbb C_\Omega^\perp $ is strictly less than $ 1 $.
	
	Given $ \phi \in \mathbb C_\Omega^\perp $, thanks to \cite[Proposition 10]{GI22}, we have
	\begin{equation}\label{limitenulo}
		\mathscr L_{\tilde f}^n \phi = \frac{1}{h} \, \frac{1}{\chi^n} \, \mathscr L_f^n (h \phi)\,\,\underrightarrow{n \to \infty}\,  \int \phi \, d\mu = 0, \qquad \textrm{(uniformly)}. 
	\end{equation}
	If we consider now $ \| \cdot \|_\Omega = \| \cdot \|_\infty + | \cdot |_\Omega $, from the compactness (with respect to the uniform topology) of
	$$ B_1(\mathbb C_\Omega^\perp) := \{ \psi \in \mathbb C_\Omega^\perp : \| \psi \|_\Omega \le 1 \}, $$
	it is easy to show that, for any $ \epsilon > 0 $, there exists a positive integer $ n_0 = n_0(\epsilon) $ such that
	$$ \big \|  \mathscr L_{\tilde f}^n \phi \big \|_\infty < \epsilon, \qquad \forall \, n \ge n_0, \, \forall \, \phi \in B_1(\mathbb C_\Omega^\perp). $$
	For $ \phi \in B_1(\mathbb C_\Omega^\perp) $, the previous proposition ensures that
	$$
	\big\vert  \mathscr L_{\tilde f}^{N+n}  \phi \big\vert_\Omega \le  \tilde\Gamma \, \big(  \tau(N) \, | \mathscr L_{\tilde f}^n  \phi  |_\Omega +  \vert\vert \mathscr L_{\tilde f}^n  \phi  \vert\vert_\infty \big)
	\le \tilde \Gamma^2 \, \big(  \tau(n) +  1 \big) \, \tau(N) + \tilde \Gamma \, \vert\vert \mathscr L_{\tilde f}^n  \phi  \vert\vert_\infty.
	$$
	Thus, as $ \Omega $ vanishes orderly, for $ n $ and $ N $ large enough, one obtains $ \|   \mathscr L_{\tilde f}^{N+n}  \phi \|_\Omega \le 2/3 $ for all $ \phi \in B_1(\mathbb C_\Omega^\perp) $.
	Since the spectral radius of $ \mathscr L_{\tilde f} |_{ \mathbb C_\Omega^\perp} $ may be compute as $ \inf_{k \ge 0}  \|  (\mathscr L_{\tilde f}  |_{ \mathbb C_\Omega^\perp})^k \|_\Omega^{1/k} $, 
	it is no larger than $ (2/3)^{1/(N+n)} $.
\end{proof}

We recapitulate the standard argument to obtain the exponential decay of correlations from the spectral gap property.
As the spectral radius of  $ \mathscr L_{\tilde f} |_{ \mathbb C_\Omega^\perp} $ is strictly smaller than $ 1$, there are constants $ \rho \in (0, 1) $ and $ K_0 > 0 $ such that 
$  \| \mathscr L_{\tilde f}^n  \psi \|_\Omega  \le K_0 \, \| \psi \|_\Omega \, \rho^n  $ for all $ \psi \in \mathbb C_\Omega^\perp $ and $ n \ge 1 $. Therefore, for $ \phi, \psi \in \mathscr C_{\Omega}(\mathbb T) $,
with $ \int \psi \, d\mu = 0 $, one has
$$ \Big | \int \phi \,\, \psi \circ T^n \, d\mu \Big | \le \| \phi \|_{L^1(\mu)}  \, \| \mathscr L_{\tilde f}^n  \psi \|_\infty \le K_0 \, \| \phi \|_{L^1(\mu)}  \,  \| \psi \|_\Omega \, \rho^n, $$  
from which Theorem~\ref{decaimentoexponencial} immediately follows.

With respect to the Central Limit Theorem (Theorem~\ref{CLT}), this can be derived from the general theorem below, 
which is an abstraction of results due to \cite{DF37, Nag57, RE83, GH88}. For a nice exposition of the proof of this theorem, watch \cite{Sar20}.

\begin{teorema}
	Let $ \mu $ be a $T$-invariant probability measure with respect to which the system is (strongly) mixing.
	Let $ \mathscr L $ be the bounded linear operator whose action on $ L^1(\mu) $ is characterized by
	$$ \int \Phi \, \cdot \, \mathscr L \Psi \, d\mu = \int \Phi \circ T \, \cdot \, \Psi \, d\mu, \quad \forall \, \Phi \in L^\infty(\mu), \, \forall \, \Psi \in L^1(\mu). $$
	Suppose that $ \mathscr L $ has the spectral gap property on a Banach space $ (\mathcal B, \| \cdot \|_{\mathcal B}) $ that contains the constants,
	is closed under multiplication, and satisfies both $ \| \varphi \, \psi \|_{\mathcal B} \le \| \varphi \|_{\mathcal B} \, \| \psi \|_{\mathcal B} $ and $ \| \varphi \|_{L^1(\mu)} \le \| \varphi \|_{\mathcal B}  $,
	for all $ \varphi, \psi \in \mathcal B $.  Then if $ \phi \in \mathcal B $ is bounded and is not cohomologous to a constant, there exists
	$ \gamma = \gamma(\phi) > 0 $ such that
	$$ \lim_{n \to \infty} \mu \Big\{  x \, : \frac{1}{\sqrt{n}} S_n\big(\phi - \int \phi \, d\mu \big)(x) < b \Big\} = \frac{1}{\gamma\sqrt{2\pi}} \int_{- \infty}^b e^{-t^2/2\gamma^2} \, dt. $$
\end{teorema}

For the sake of completeness, we include the only aspect that has not even been implicitly discussed so far.
As for Lemma~\ref{teta}, the next proposition is an exclusive consequence of the $T$-compatibility between moduli, without the intervention of any additional assumption used in the central results of this note.

\begin{proposition}
	The dynamical system $ (\mathbb T, T) $ is strongly mixing with respect to the Gibbs-equilibrium state $ \mu $ associated with the potential $ f \in \mathscr C_{\omega}(\mathbb T)$.
\end{proposition}

\begin{proof}
	It suffices to show that $ \int \phi \circ T^n \, \psi \, d\mu \to \int \phi \, d\mu \, \int \psi \, d\mu $ as $ n \to \infty $, for any $ \phi, \psi \in L^2 (\mu) $. 
	However, given $ \phi, \psi \in C(\mathbb T) $, with $ \int \psi \, d\mu = 0 $, by the same reason as in~\ref{limitenulo}, one has
	$$ \Big | \int \phi\circ T^n \, \psi \, d\mu \Big | =  \Big | \int \phi \, \mathscr L_{\tilde f}^n \psi \, d\mu \Big | \le \| \phi\|_ {L^1(\mu)} \, \| \mathscr L_{\tilde f}^n \psi \|_\infty   \Rightarrow 0, n \to \infty. $$
	Thus the conclusion follows from the denseness of continuous functions in $ L^2(\mu) $.
\end{proof}

\section{Illustrations}

We gather in this section some examples of applications of our results. 
We compile both innovative contributions and situations already recorded in the literature.

\paragraph{A slowly varying scenario.}  
In \cite{GI22}, we have considered a dynamics described for $ x $ small enough as 
$$ T_k(x) = x \big( 1 + A_k (\log^k 1/x)^{-1} \big) \mod 1, $$
where $ A_k > 0 $ is a constant and $ \log^k $ stands for the $k$-times composition of the logarithm function. 
Taking into account moduli defined in a neighborhood of the origin as  
$ \omega_0(x) = (\log^k 1/x)^{-1} (\log 1/x)^{-1} (\log^2 1/x)^{-2} $ and $ \Omega_0(x) = (\log^2 1/x)^{-1} $, we have showed via condition~\ref{condicao_suficiente} that a Ruelle-Perron-Frobenius theorem holds. 
Since it is a calculus exercise to check that, for any fixed $ s \in (0, 1) $, both 
$$ \omega_s(x) = x^s (\log^k 1/x)^{-1} (\log 1/x)^{-1} (\log^2 1/x)^{-2}
\quad \textrm{ and } \quad
\Omega_s(x) = x^s  (\log^2 1/x)^{-1} $$
are concave in a neighborhood of the origin, from Proposition~\ref{ganho} we have the following illustrative result.

\begin{proposition}
	Given a positive integer $ k $ and $ s \in (0,1) $, regarding the dynamical system $ T_k $ above,
	for any potential $ f \in \mathscr C_{\omega_s}(\mathbb T) $, there exists a unique associated Gibbs-equilibrium state $ \mu $ which has exponential decay of correlations and satisfies the Central Limit Theorem with respect to the class $ \mathscr C_{\Omega_s}(\mathbb T) $. 
\end{proposition}

\paragraph{A central family.}
Perhaps the most relevant class of examples is formed by the Manneville-Pomeau maps,
whose analysis is connected to the mathematical modeling of intermittency \cite{Man80,PM80}.
These maps are defined as
$$ M_q(x) = x (1 + x^q) \mod 1, \qquad \textrm{with } q > 0. $$
By using condition~\ref{condicao_suficiente}, it is not difficult to formulate a general result concerning this family. 

\begin{proposition}
	Suppose $ \theta : (0, 1] \to [0, +\infty) $ is a continuous function such that 
	(in a neighborhood of the origin)
	$ x \mapsto x^P \theta(x) $ is concave and non-decreasing as well as 
	$ \lim_{x \to 0^+} x^P \sup_{0 < \mathsf d < 1/2} \frac{\theta ({\mathsf d} x )}{\theta (\mathsf d)} = 0 $ 
	for some constant $ P > 0 $. 
	Assume also that $ p > P $ and $ 0< q \le p - P $ satisfy 
	$ \liminf_{x \to 0^+} x^{q - p + P} \big( \frac{\theta((1+c)x)}{\theta(x)}(1+c)^P - 1\big) > 0 $ for any $ c > 0 $ sufficiently small.
	Denote
	$$ \omega(x):=x^p \theta(x) \qquad \textrm{and} \qquad \Omega(x):=x^P \theta(x). $$
	Then, with respect to the Manneville-Pomeau map $ M_q $, 
	given a potential $ f \in \mathscr C_{\omega}(\mathbb T) $, 
	the transfer operator $ \mathscr L_f $ acting on $ \mathscr C_{\Omega}(\mathbb T) $ satisfies a Ruelle-Perron-Frobenius theorem
	with spectral gap.
	In particular, the unique associated Gibbs-equilibrium state has exponential decay of correlations and satisfies a Central Limit Theorem. 
\end{proposition}

When looking for a systematized statement as above, we realized that we had unnecessarily assumed in our study \cite{GI22} the concavity of $\omega$. 
Actually, the main results on the existence and uniqueness of Gibbs-equilibrium states rest exclusively on $T$-compatibility.  
Only the concavity of $ \Omega $ is used once in all that work: during the proof of \cite[Proposition 7]{GI22} to show that~\ref{condicao_suficiente} provides a sufficient condition for $T$-compatibility.
In order to completely avoid concavity in that article, in a reformulation of the aforementioned proposition, condition~\ref{condicao_suficiente} could be replaced by the broader hypothesis
$$ \liminf_{x \to 0^+} \frac{\Omega\big(x\big(1+cV(x)\big)\big) - \Omega\big(x\big)}{\omega(x)} > 0 \qquad \textrm{for $c>0$ small enough}, $$
as the reader can easily see when analyzing the key inequalities in that proof.
However, the concavity of $ \Omega $ is used in this note.

Regarding the result, note that for $ \theta \equiv 1 $ we recover H\"older moduli of continuity.
The conclusions thus obtained form part of results already present in the literature (see, for instance, \cite{LR14a,LR14b,Klo20}). For $ \theta(x) = 1 + | \log x | $ in turn, we deal with locally H\"older continuous functions with respect to which we have not identified known prior results.

\paragraph{An inquiry suggestion.}
Karamata theory (see \cite{Sen76} for details) states that $ V $ has the form $ V(x) = x^\sigma L(x) $, where the function
$ L $ satisfies $ \lim_{x \to 0} L(tx) / L(x) = 1 $ for all $ t > 0 $. Let us focus on the case $ \sigma \in [0, 1) $.
If we thus consider the moduli of continuity $ \omega(x) = x L(x) $ and $ \Omega(x) = x^\tau $, with $ 0 < \tau \le 1 - \sigma $,
clearly
$$ \frac{V(x)}{\omega(x)} \big( \Omega\big((1+c)x\big) - \Omega\big(x\big) \big) = x^{\sigma - 1 + \tau} \big( (1+ c)^\tau - 1 \big), $$
and the sufficient condition for $T$-compatibility is guaranteed. We have thus the following result.  

\begin{proposition} \label{sigma menor que 1}
	For the dynamics $ T(x) = x(1 + V(x)) \mod 1 $, $ x \in [0, 1) $, if $ V $ is regularly varying with $ \sigma \in [0, 1) $,
	then H\"older observables of H\"older exponent at most $ 1 - \sigma $ exhibit exponential decay of
	correlations with respect to equilibrium states associated with potentials that have $ \omega(x) = x^{1-\sigma}V(x) $ as modulus
	of continuity.
\end{proposition}

Perhaps the meaning of this proposition should be evaluated in light of the results of Holland \cite{Hol05} on subexponential mixing rates for intermittent maps of the circle. Notably, \cite[Theorem 2]{Hol05} details 
special situations, such as logarithmic case and intermediate logarithmic case, for decays of correlations of
H\"older potentials (of any exponent) with respect to a $T$-invariant, absolutely continuous and physical measure.
Proposition~\ref{sigma menor que 1} then seems to suggest\footnote{Note that the equilibrium states of Proposition~\ref{sigma menor que 1} are also Gibbs measures \cite[Proposition~16]{GI22}, and therefore give positive mass to every open set. Nevertheless, the fact that $ V $ is taken as increasing on $ [0, 1] $ implies that $ T $ is not differentiable over the manifold $ \mathbb T $ at the indifferent fixed point. In particular, this precludes invoking an immediate link between equilibrium measures and physical measures.} that, by restricting to a subspace of  H\"older observables with properly bounded exponent, contrary to the slow decay rate established for H\"older observables in general, the respective correlations would decay exponentially fast to zero. This is a subject to be investigated more closely. 
A quantitative understanding of the spectra of transfer operators could open a way forward. See \cite{BKL22} for an interesting discussion on this topic in different dynamic contexts.


	\footnotesize{
		
	}
\end{document}